\documentclass[11pt]{article} 
\usepackage[margin=1.2in]{geometry}
\usepackage[utf8]{inputenc}
\usepackage[T1]{fontenc}
\usepackage[english]{babel}
\usepackage{lmodern} 
\usepackage{graphicx}
\usepackage{wrapfig}
\usepackage{subfig}
\usepackage{caption}
\usepackage{paralist}
\usepackage{microtype}
\usepackage{enumitem}

\graphicspath{{./figures/}}

\newcommand{\titel}{Compact Cactus Representations of all Non-Trivial Min-Cuts}

\usepackage{algorithm} 
 
\usepackage[noend]{algpseudocode}

\usepackage[usenames]{color} 
\definecolor{hellblau}{rgb}{0.2,0.4,1} 
\definecolor{dunkelblau}{rgb}{0,0,0.8}
\definecolor{dunkelgruen}{rgb}{0,0.5,0}
\usepackage[
pdftex,
colorlinks,
linkcolor=dunkelblau,
urlcolor=dunkelblau,
citecolor=dunkelgruen,
bookmarks=true,
linktocpage=true,
pdftitle={\titel},
pdfauthor={},
pdfsubject={}
]{hyperref} 
\usepackage{pdfpages}

\usepackage{amsmath} 
\usepackage{amsthm} 
\usepackage{amsfonts} 
\usepackage{amssymb}
\theoremstyle{plain} 
\newtheorem{satz}{Satz} 
\newtheorem{theorem}[satz]{Theorem}
\newtheorem{lemma}[satz]{Lemma}

\newtheorem{corollary}[satz]{Corollary}
\theoremstyle{remark} 
 
\theoremstyle{definition}

\newtheorem{innercustomgeneric}{\customgenericname}
\providecommand{\customgenericname}{}
\newcommand{\newcustomtheorem}[2]{%
  \newenvironment{#1}[1]
  {%
   \renewcommand\customgenericname{#2}%
   \renewcommand\theinnercustomgeneric{##1}%
   \innercustomgeneric
  }
  {\endinnercustomgeneric}
}
\newcustomtheorem{customtheorem}{Theorem}
\newcustomtheorem{customlemma}{Lemma}
\newcustomtheorem{customproposition}{Proposition}
\newcustomtheorem{customcorollary}{Corollary}

\begin{document}
	\title{\titel}
	\author{On-Hei S.\ Lo\thanks{This research is supported by the grant SCHM 3186/1-1 (270450205) from the Deutsche Forschungsgemeinschaft (DFG, German Research Foundation).}\\Institute of Mathematics\\TU Ilmenau, Germany
		\and Jens M.\ Schmidt\addtocounter{footnote}{-1}\footnotemark\\Institute of Mathematics\\TU Ilmenau, Germany
		\and Mikkel Thorup\thanks{Mikkel Thorup's research is supported by his Advanced Grant DFF-0602-02499B from the Danish Council for Independent Research and by his Investigator Grant 16582, Basic Algorithms Research Copenhagen (BARC), from the VILLUM Foundation.}\\Department of Computer Science\\University of Copenhagen, Denmark}
	\date{}
	\maketitle

	\begin{abstract}
		Recently, Kawarabayashi and Thorup presented the first deterministic edge-connectivity recognition algorithm in near-linear time. A crucial step in their algorithm uses the existence of vertex subsets of a simple graph $G$ on $n$ vertices whose contractions leave a multigraph with $\tilde{O}(n/\delta)$ vertices and $\tilde{O}(n)$ edges that preserves all non-trivial min-cuts of $G$, where $\delta$ is the minimum degree of $G$ and $\tilde{O}$ hides logarithmic factors.
		
		We present a simple argument that improves this contraction-based sparsifier by eliminating the poly-logarithmic factors, that is, we show a contraction-based sparsification that leaves $O(n/\delta)$ vertices and $O(n)$ edges, preserves all non-trivial min-cuts and can be computed in near-linear time $\tilde{O}(m)$, where $m$ is the number of edges of $G$. We also obtain that every simple graph has $O((n/\delta)^2)$ non-trivial min-cuts.
		
		Our approach allows to represent all non-trivial min-cuts of a graph by a cactus representation, whose cactus graph has $O(n/\delta)$ vertices. Moreover, this cactus representation can be derived directly from the standard cactus representation of all min-cuts in linear time. We apply this compact structure to show that all min-cuts can be explicitly listed in $\tilde{O}(m) + O(n^2 / \delta)$ time for every simple graph, which improves the previous best time bound $O(nm)$ given by Gusfield and Naor.
	\end{abstract}
	

	\section{Introduction}
	Edge-connectivity and the structure of (near-)minimum cuts of graphs have been studied intensively for the last 60 years and proved to have a wide range of real-world applications such as network reliability and information retrieval. Many of the discovered structures like Gomory-Hu trees~\cite{Gomory1961}, cactus representations~\cite{Dinits1976} and the lattice of minimum $s$-$t$-cuts led to increasingly faster algorithms for recognizing, listing or counting various (near-)minimum cuts of graphs. These structures fall into the field of \emph{graph sparsifiers}, which decrease the graph size while preserving certain connectivity properties.
	
	In this paper, we demonstrate that a very simple contraction argument on cactus representations gives a sparsifier preserving non-trivial min-cuts that surpasses the best ones known, where a cut is \emph{trivial} if either it or its complement is a singleton. It turns out one can apply it to enumerate min-cuts efficiently.

	\subsection{Previous Work}
	Recently, Kawarabayashi and Thorup~\cite{Kawarabayashi2018} presented the first deterministic min-cut algorithm with near-linear running time $O(m \log^{12}n)$ for simple graphs. They showed that vertex sets of the input graph can be determined and contracted in near-linear time such that the remaining graph has only $\tilde{O}(n/\delta)$ vertices, $\tilde{O}(n)$ edges and preserves all non-trivial min-cuts of the original graph~\cite[Theorem~1.3]{Kawarabayashi2018}. This contraction-based sparsification implies a min-cut algorithm in near-linear time, as the contractions leave a graph on which Gabow's algorithm~\cite{Gabow1995} can be applied, which runs itself in time $\tilde{O}(\lambda m)$, where $\lambda$ denotes the edge-connectivity. Subsequently, Henzinger, Rao and Wang~\cite{Henzinger2017} improved the running time of this min-cut algorithm to $O(m \log^2 n \log \log^2 n)$ by replacing its diffusion subroutine with a flow-based one; this algorithm relies also on contraction-based sparsification. We will first focus on the question to which extent such contraction-based sparsifiers can be improved.
	
	Cactus representation is well-known for its capability of storing min-cuts in a compact way. It was first shown by~\cite{Dinits1976} that one can construct a cactus representation of $O(n)$ size which represents all min-cuts. Later, Nagamochi and Kameda~\cite{Nagamochi1994} gave two more forms of cactus representation with the same size bound. We will introduce a cactus representation with a better size bound, which represents all non-trivial min-cuts.
	
	How fast one can enumerating all min-cuts comes up as a natural question. For instance, counting the number of min-cuts can be seen as a measure for network reliability. Gusfield and Naor~\cite{Gusfield1993} showed that all min-cuts can be listed explicitly in $O(nm)$ time for every simple graph. We will apply the new cactus representation to improve this time bound.
	
	\subsection{Our Results}
	We give an asymptotically optimal improvement of the bounds of the contraction-based sparsifier of Kawarabayashi and Thorup~\cite[Theorem~1.1]{Kawarabayashi2018} by eliminating its poly-logarithmic factors. Hence, every simple graph can be sparsified in near-linear time by contractions of vertex subsets such that $O(n/\delta)$ vertices and $O(n)$ edges are left and every non-trivial min-cut is preserved. We provide not only the optimal bounds but also an insight into how these non-trivial min-cuts can be represented by a specific cactus representation.
	
	For a graph $G$, let $\mathcal{C}(G)$ be the set of all min-cuts of $G$ and let $\mathcal{NC}(G)$ be the set of all non-trivial min-cuts of $G$. We need to generalize cactus representations to represent proper subsets of $\mathcal{C}(G)$ instead of the usual set $\mathcal{C}(G)$: For a subset $\mathcal{S}$ of $\mathcal{C}(G)$, a \emph{cactus representation for $\mathcal{S}$} represents every min-cut in $\mathcal{S}$ by some min-cut of a cactus graph, but not necessarily all such min-cuts (see the next section for a precise definition). In particular, a cactus representation for $\mathcal{C}(G)$ is the usual cactus representation.
	
	We will first show that a cactus representation for $\mathcal{C}(G)$ can be transformed to a cactus representation for $\mathcal{NC}(G)$ whose cactus graph has only $O(n/\delta)$ vertices (Theorem~\ref{thm:bb}). We will then prove that simply contracting the vertex sets that correspond to the cactus vertices in the new cactus representation gives our main result, the sparsifier mentioned above (Theorem~\ref{thm:main}).
	
	\begin{theorem} \label{thm:bb}
		Given a simple graph $G$ on $n$ vertices and minimum degree $\delta$, a cactus representation for $\mathcal{C}(G)$ can be transformed in linear time to a cactus representation for $\mathcal{NC}(G)$ whose cactus graph has $O(n/\delta)$ vertices.
	\end{theorem}
	
	We remark that by detailed calculation it can be shown that the cactus graph mentioned above has less than $30n / \delta$ vertices.
	
	\begin{theorem} \label{thm:main}
		Let $G$ be a simple graph on $n$ vertices, $m$ edges and minimum degree $\delta$. Then vertex sets of $G$ can be computed and contracted in time $\tilde{O}(m)$ such that the remaining multigraph has $O(n/\delta)$ vertices, $O(n)$ edges and preserves all non-trivial min-cuts of $G$.
	\end{theorem}

	Kawarabayashi and Thorup~\cite{Kawarabayashi2018} showed that every connected simple graph has $\tilde{O}((n/\delta)^2)$ non-trivial minimum cuts for some constant $c$. This follows directly from the well-known result in~\cite{Dinits1976} that the number of minimum cuts in any connected graph $H$ is at most $O(|V(H)|^2)$, when taking $H$ as the sparsified graph. Hence, Theorem~\ref{thm:main} implies the following fundamental new bound.
	
	\begin{corollary} \label{cor:maincor}
		Every connected simple graph $G$ has $O((n/\delta)^2)$ non-trivial min-cuts and hence $n + O((n/\delta)^2)$ min-cuts.
	\end{corollary}
	
	We will show in Section~\ref{sec:tightness} that Theorems~\ref{thm:bb}+\ref{thm:main} and also Corollary~\ref{cor:maincor} are asymptotically optimal from various perspectives; all of them improve the best known results so far, whenever $\delta$ is superconstant.
	
	We apply the compact cactus representation to explicitly list all min-cuts, each as a list of the 
	cut edges. The previous best time bound of $O(nm)$ for enumerating all min-cuts of a simple graph was given by Gusfield and Naor~\cite{Gusfield1993}. We have the following improvement:
	
	\begin{theorem} \label{thm:enumeration}
		All min-cuts of a simple graph $G$ can be listed explicitly in $\tilde{O}(m) + O(n^2 / \delta)$ time, with output size $O(\lambda (n+(n/\delta)^2))=O(m+n^2/\delta)$.
	\end{theorem}
	
	We emphasize that, as in~\cite{Gusfield1993}, instead of vertex subset of the graph, a cut in the output is represented by the $\lambda$ edges going across some vertex subset. We do so because the vertex subsets representing min-cuts may have average size $\Theta(n)$ per each, e.g. a cycle.  
	
	The space bound of the explicit representation of all min cuts is tight as  there are simple graphs with  $\Theta(n^2 / \delta^2)$ min-cuts and where $\lambda = \Theta(\delta)$ (see Section~\ref{sec:tightness}). Moreover, for the $m$ term, we can easily find simple graphs where all edges are in some min-cut, e.g. a regular expander where the min-cuts are exactly the trivial cuts. 

	\subsection{Technical Overview}
	We will use a variant of the well-known cactus representation that is restricted to non-trivial min-cuts to find the vertex sets that have to be contracted. This cactus representation can be derived from any standard cactus representation that represents all min-cuts in linear time. Since Kawarabayashi and Thorup showed that the standard cactus representation can be found deterministically in near-linear time~\cite{Kawarabayashi2018} (alternatively, one may use the randomized Monte-Carlo algorithm in~\cite{Karger2009}), this gives a deterministic near-linear running time $\tilde{O}(m)$ to find the vertex sets to be contracted.
	
	In order to enumerate all min-cuts of a simple graph, we combine the framework of Gusfield and Naor~\cite{Gusfield1993} with our new cactus representation. We wish to apply their method to list the min-cuts on the sparsified graph obtained from vertex subset contractions. However this graph may have multiple edges. Thus we apply their algorithm in a cactus graph instead, and place back the edges of the original graph when we start listing the min-cuts.
	
	\section{Preliminaries}
	All graphs considered in this paper are non-empty and finite. Let $G:=(V,E)$ be a graph. We denote by $uv$ the edge with endvertices $u, v$ if $G$ is a undirected graph, and by $uv$ or $(u, v)$ the edge directed from $u$ to $v$ if $G$ is a directed graph. \emph{Contracting} a vertex subset $X \subseteq V$ consists in identifying all vertices in $X$ and deleting occurring self-loops but no parallel edges (we do not require that $X$ induces a connected graph in $G$). For an edge $vw$ in $G$, \emph{contracting $vw$} means contracting $\{v, w\}$.
	
	For non-empty and disjoint vertex subsets $X,Y \subset V$, let $E_G(X,Y)$ denote the set of all edges in $G$ that have one endvertex in $X$ and one in $Y$. Let further $\overline{X} := V-X$, $d_G(X,Y):=|E_G(X,Y)|$ and $d_G(X):=|E_G(X,\overline{X})|$; if $X=\{v\}$ for some vertex $v\in V$, we simply write $E_G(v,Y)$, $d_G(v,Y)$ and $d_G(v)$, respectively. A subset $\emptyset \neq X \subset V$ of a graph $G$ is called a \emph{cut} of $G$ of size $d_G(X)$. A cut $X$ of $G$ is said to be \emph{trivial} if $|X|=1$ or $|\overline{X}|=1$, and \emph{non-trivial} otherwise. By an \emph{edge cut} we mean a set $E_G(X)$ of edges for some cut $X \subset V$; we identify an edge cut $E_G(X)$ with the cut $X$ if there is no ambiguity arises. A family of cuts $X_1, \dots, X_k  \subset V$ are \emph{uncrossing} if we have $X_i \subset X_j$ or $X_i \supset X_j$ for all $i \neq j \in \{1, \dots, k\}$.
	
	Let the \emph{length} of a path (cycle) be the number of its edges; a $k$-\emph{cycle} is a cycle of length $k$. Let $\delta(G) := \min_{v\in V} d_G(v)$ be the \emph{minimum degree} of $G$.
	
	A \emph{block} of $G$ is a maximal subgraph of $G$ that does not contain a \emph{cut-vertex} (i.e. a separator made of a single vertex). For two vertices $v,w\in V$, a \emph{$v$-$w$-cut} is a vertex set $X\subseteq V$ such that exactly one of $\{v,w\}$ is in $X$. Let $\lambda_G(v,w)$ be the minimum $d_G(X)$ over all $v$-$w$-cuts $X$. Two vertices $v,w\in V$ are called \emph{$k$-edge-connected} if $\lambda_G(v,w) \geq k$. The \emph{edge-connectivity} $\lambda := \lambda(G)$ of $G$ is $\min_{v,w\in V} \lambda_G(v,w)$. We omit subscripts whenever the graph is clear from the context.
	
	We call a multigraph $\mathcal{K}$ a \emph{cactus} if it is 2-edge-connected, contains no self-loops, and every edge in $\mathcal{K}$ belongs to exactly one cycle (which may be of length 2, i.e. a pair of parallel edges). In other words, all blocks of $\mathcal{K}$ are cycles. This way, an edge cut in $\mathcal{K}$ is a min-cut if and only if it is consists of two edges from a cycle in $\mathcal{K}$. 
	
	Let $\mathcal{K}$ be a cactus and $\varphi$ be a mapping from $V(G)$ to $V(\mathcal{K})$. Given $\mathcal{S} \subseteq \mathcal{C}(G)$, we say that $(\mathcal{K}, \varphi)$ is a \emph{cactus representation} of $G$ \emph{for} $\mathcal{S}$ if (i) for every $X \in \mathcal{S}$, there is a min-cut $Y$ in $\mathcal{K}$ with $X = \varphi^{-1}(Y)$ and (ii) for every min-cut $Y$ in $\mathcal{K}$, $\varphi^{-1}(Y)$ is a min-cut in $G$. A vertex $v$ in $\mathcal{K}$ is \emph{empty} if $\varphi^{-1}(v)$ is empty, a \emph{singleton} if $\varphi^{-1}(v)$ consists of exactly one vertex of $G$, and a \emph{$k$-junction} if $v$ is contained in exactly $k$ cycles of $\mathcal{K}$. It has been proven by Dinits et al.~\cite{Dinits1976} that every graph $G$ admits a cactus representation for $\mathcal{C}(G)$. A cactus representation $(\mathcal{K}, \varphi)$ of $G$ for $\mathcal{S}$ is \emph{minimal} if no smaller cactus representation for $\mathcal{S}$ can be obtained by contracting an edge of $\mathcal{K}$ and revising $\varphi$ accordingly.

	A \emph{directed acyclic graph} $\mathcal{A}$ is a directed multigraph with no directed cycle. Given $u, v \in V(\mathcal{A})$, we say $u$ is a \emph{predecessor} (\emph{successor}) of $v$ if there is a path directed from $u$ to $v$ ($v$ to $u$) in $\mathcal{A}$. In this paper we consider only directed acyclic graphs $\mathcal{A}$ each has (unique) vertices $s_\mathcal{A}, t_\mathcal{A} \in V(\mathcal{A})$ such that $s_\mathcal{A}$ and $t_\mathcal{A}$ are a successor and a predecessor of $v$ for any $v \in V(\mathcal{A})$, respectively. For $\emptyset \neq X \subset V(\mathcal{A})$, we call $X$ a \emph{closed set} of $\mathcal{A}$ if for every $v$ in $X$, every successor of $v$ is also in $X$. It is clear that $s_\mathcal{A} \in X$ and $t_\mathcal{A} \notin X$ for any closed set $X$ of $\mathcal{A}$. A \emph{directed acyclic graph representation} (DAG) $(\mathcal{A}, \rho)$ of $G$ consists of a directed acyclic graph $\mathcal{A}$ and a mapping $\rho$ from $V(G)$ to $V(\mathcal{A})$. We also identify a DAG $(\mathcal{A}, \rho)$ with $\mathcal{A}$ when we describe $V(\mathcal{A})$ as a partition of $V(G)$. For a closed set $X$ of $\mathcal{A}$, we say $\rho^{-1}(X)$ is represented by $(\mathcal{A}, \rho)$.

	\section{Contraction-Based Sparsification}
	
	Let $G := (V, E)$ be a simple graph and $\mathcal{S} \subseteq \mathcal{C}(G)$. The following lemmas concern a cactus representation $(\mathcal{K}, \varphi)$ of $G$ for $\mathcal{S}$.

	For a vertex $v$ of a cycle $C$ in $\mathcal{K}$ and its two incident edges $e$ and $f$ in $C$, let $\mathcal{K}[C,v]$ be the component of $\mathcal{K}-e-f$ that contains $v$.
	
	\begin{lemma}[\cite{Dinits1976}] \label{lem:edgesbetweenneighbors}
		Let $u$ and $v$ be two distinct vertices in a cycle $C$ of length at least three in $\mathcal{K}$. If $u$ and $v$ are neighbors in $C$, then $G$ has exactly $\lambda /2$ edges between $\varphi^{-1}(\mathcal{K}[C, u])$ and $\varphi^{-1}(\mathcal{K}[C, v])$. Otherwise, $G$ has no edge between $\varphi^{-1}(\mathcal{K}[C, u])$ and $\varphi^{-1}(\mathcal{K}[C, v])$.
	\end{lemma}
	\begin{proof}
		Suppose $u$ and $v$ are neighbors in $C$. Let $X_1 := \varphi^{-1}(\mathcal{K}[C, u])$, $X_2 := \varphi^{-1}(\mathcal{K}[C, v])$ and $X_3 := V - X_1 - X_2$. As $(\mathcal{K}, \varphi)$ is a cactus representation, $X_1$, $X_2$ and $X_3$ are min-cuts in $G$, respectively. This implies that $d(X_1, X_2) + d(X_1, X_3) = d(X_2, X_3) + d(X_2, X_1) = d(X_3, X_1) + d(X_3, X_2) = \lambda$. Hence, $d(X_1, X_2) = \lambda/2$.
		
		If $u$ and $v$ are not adjacent in $C$, then there exist vertices $w_1, w_2 \neq v$ adjacent to $u$ in $C$. Let $X_1 := \varphi^{-1}(\mathcal{K}[C, u])$, $X_2 := \varphi^{-1}(\mathcal{K}[C, v])$ and $X_{i + 2} := \varphi^{-1}(\mathcal{K}[C, w_i])$ for $i = 1, 2$. By the first case, we have that $d(X_1, X_3) = d(X_1, X_4) = \lambda/2$. Note that $\lambda = d(X_1) \geq \sum_{i = 2, 3, 4} d(X_1, X_i) = \lambda + d(X_1, X_2)$. Hence we have $d(X_1, X_2) = 0$.
	\end{proof}

	\begin{lemma} \label{lem:no1junctionneighbors}
		Suppose that $\delta(G) \geq 3$. Then no cycle in $\mathcal{K}$ contains two adjacent 1-junction singletons. In particular, at most half of the vertices of every cycle in $\mathcal{K}$ are 1-junction singletons.
	\end{lemma}
	\begin{proof}
		Assume to the contrary that $\{v\} \subseteq V$ and $\{w\} \subseteq V$ are the preimages of two adjacent 1-junction singletons of a $k$-cycle in $\mathcal{K}$ (for some $k \geq 2$). In particular, $v$ has degree $\lambda$ in $G$, as it is the preimage of a singleton. This implies $\lambda = \delta \geq 3$, as there is no 1-junction singleton when $\lambda < \delta$.
		
		If $k = 2$, $\mathcal{K}$ and thus also $G$ contains exactly two vertices, which contradicts $\lambda \geq 3$, as $G$ is simple. Otherwise, $k \geq 3$ and $G$ contains exactly $\lambda/2 > 1$ parallel edges between $v$ and $w$ by Lemma~\ref{lem:edgesbetweenneighbors}, which contradicts that $G$ is simple.
	\end{proof}
	
	The following lemma holds trivially from the definition of cactus representation.
	
	\begin{lemma} \label{lem:atleastone}
		$|\varphi^{-1}(\mathcal{K}[C, v])| \geq 1$ for any cycle $C$ in $\mathcal{K}$ and any vertex $v$ in $C$.
	\end{lemma}
	
	\subsection{Proofs of Theorems~\ref{thm:bb} and~\ref{thm:main}}
	
	Now let $(\mathcal{K},\varphi)$ be a cactus representation of $G$ for the set $\mathcal{C}(G)$ of all min-cuts of $G$ and consider the following four simple modifications. Note that when an edge $vw$ of the cactus $\mathcal{K}$ is contracted into a new vertex $v'$, the mapping $\varphi$ will be revised accordingly, namely all vertices in $\varphi^{-1}(v)$ and $\varphi^{-1}(w)$ will be mapped to $v'$.
	
	\begin{enumerate}[label=(\roman*)] 
		\item For any 2-cycle that contains a 1-junction singleton $v$ and another vertex $a$, contract $va$.
		\item For any 3-cycle that contains a 1-junction singleton $v$ and two other vertices $a$ and $b$, delete the edge $ab$ and add the edges $av$ and $bv$.
		\item For any 2-cycle that contains an empty 2-junction $v$ and another vertex $a$, contract $va$.
		\item For any 3-cycle that contains two empty 2-junctions $v, w$, contract $vw$.
	\end{enumerate}
	
\begin{figure}[h!t]
	\centering
	\subfloat[A graph $G$ (solid lines) and a cactus representation for $\mathcal{C}(G)$ (dotted lines). Modification~(i) is applicable for $v$ and $a$, and Modification~(ii) for the singleton $w$, $a$ and $b$.]{%
		\includegraphics[scale=0.7]{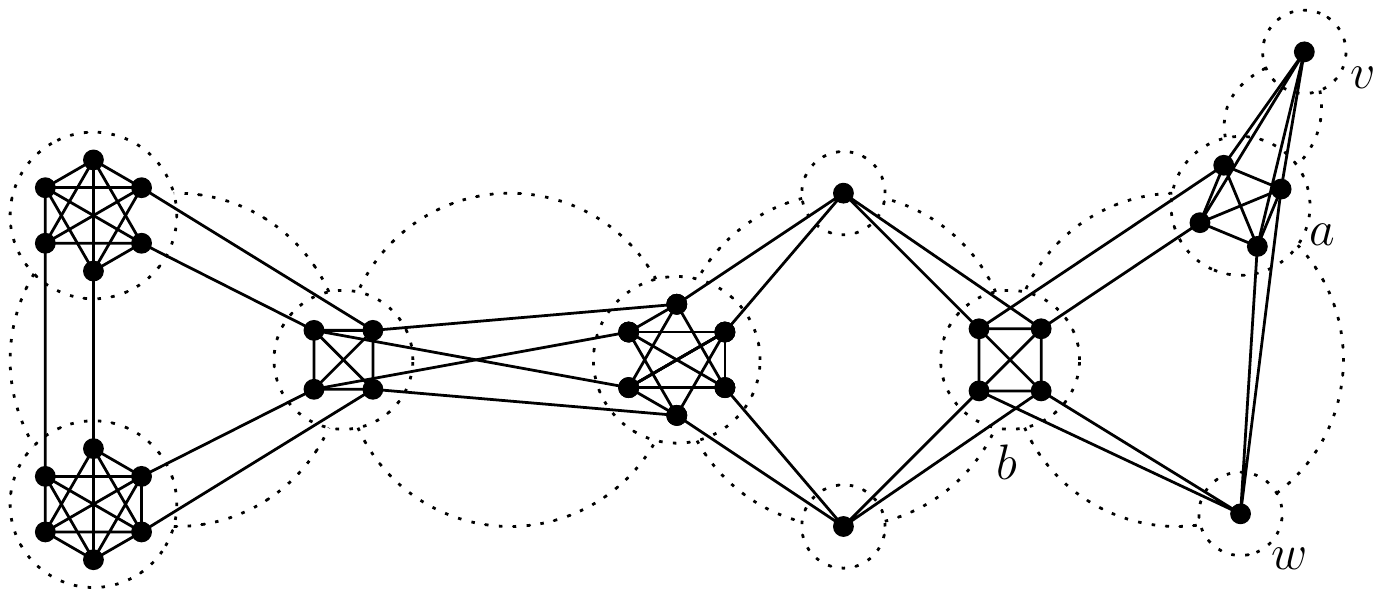}
		\label{fig:example0}
	}
	\\
	\subfloat[The cactus representation $(\mathcal{K'}, \varphi')$ for $\mathcal{NC}(G)$ after applying Modifications~(i) and~(ii) as long as possible.]{
		\includegraphics[scale=0.7]{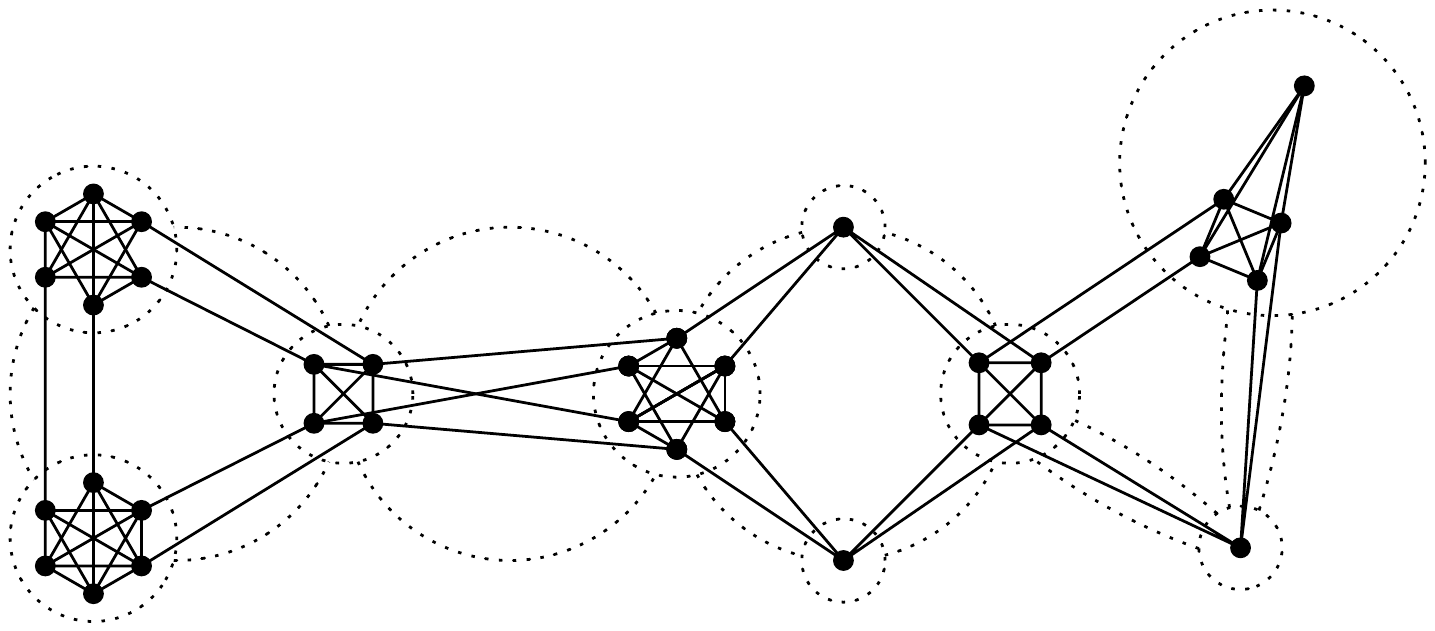}
		\label{fig:example1}
	}
	\\
	\subfloat[The xylem $\mathcal{X}$ of $(\mathcal{K'}, \varphi')$, in which circle vertices depict the center vertices. The vertices $l_1$ and $l_2$ are deleted in order to obtain $\mathcal{X'}$.]{
		\includegraphics[scale=0.7]{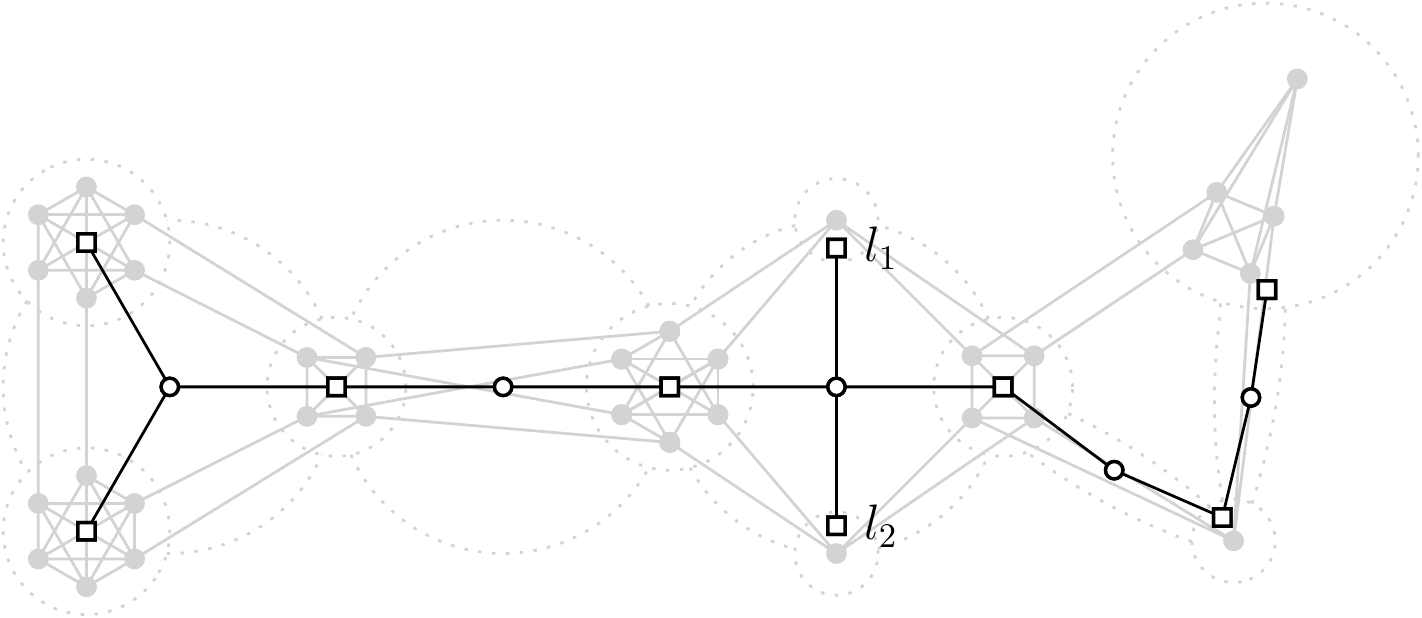}
		\label{fig:example2}
	}
	\caption{A graph $G$ satisfying $\lambda = \delta = 4$ and cactus representations of $G$.}
	\label{fig:example}
\end{figure}
	
	By applying these modifications iteratively to $(\mathcal{K},\varphi)$ as long as possible, we obtain a cactus representation $(\mathcal{K}',\varphi')$ (see Figure~\ref{fig:example}). Then $(\mathcal{K}',\varphi')$ is a cactus representation for the set of all non-trivial min-cuts $\mathcal{NC}(G)$, as every application of Modification~(i) and~(ii) destroys only a trivial min-cut $\{v\}$ (in particular,~(ii) preserves the possibly non-trivial min-cuts that separate $a$ and $b$) and that of Modifications~(iii) and~(iv) do not destroy any min-cut. Thus, $(\mathcal{K'}, \varphi')$ represents a set of min-cuts $\mathcal{R}$ such that $\mathcal{NC}(G) \subseteq \mathcal{R} \subseteq \mathcal{C}(G)$.
	
	Note that, by Lemma~\ref{lem:atleastone} and the exhaustive application of Modification~(i), if $|\varphi^{-1}(\mathcal{K}'[C, v])| = 1$ for some cycle $C$ in $\mathcal{K}'$ and some vertex $v$ in $C$, then $v$ must be a 1-junction singleton.
	
	We claim that the cactus representation $(\mathcal{K}',\varphi')$ for $\mathcal{NC}(G)$ is minimal. Suppose not, and let $vw$ be an edge in a cycle $C$ of $\mathcal{K}'$ such that we can get a smaller cactus representation by contracting $vw$. If $C$ is of length greater than 3, then there is another edge $st$ of $C$ such that $vw$ and $st$ share no common endvertex and there is a min-cut $X$ separating $v, s$ from $w, t$ in $\mathcal{K}'$. By Lemma~\ref{lem:atleastone}, $\varphi^{-1}(X)$ is a non-trivial cut of $G$, and this cut will not preserved if we contract the edge $vw$. Thus we can assume $C$ is of length either 2 or 3. If $C$ is of length 3, then $|\varphi^{-1}(\mathcal{K}'[C, v])| > 1$; otherwise $v$ is a 1-junction singleton, contradicting the exhaustive application of Modification~(ii). If $v$ is non-empty or empty but not a 2-junction, then the non-trivial min-cut $\varphi^{-1}(\mathcal{K}'[C, v])$ of $G$ cannot be preserved if we contract $vw$. Therefore $v$ and $w$ must be empty 2-junctions, which is however not possible because of Modification~(iv). If $C$ is of length 2, then, similarly, one of $v, w$ must be an empty 2-junction; this possibility has been ruled out by Modification~(iii). Hence the claim is proved.
	
	Since every application of Modification~(i),~(iii) and~(iv) decreases the number of cactus vertices by 1, and every application of Modification~(ii) decreases the number of 3-cycles in the cactus by 1, $(\mathcal{K}',\varphi')$ can be computed in a running time that is linear in the size of $\mathcal{K}$. We remark that we do not need the full strength of the minimality; Modification~(iv) is actually not needed to prove our results.
	
	Kawarabayashi and Thorup~\cite[Theorem~8.2]{Kawarabayashi2018} showed that a standard cactus representation can be found in $\tilde{O}(m)$ time. We conclude that $(\mathcal{K}',\varphi')$ can be computed in $\tilde{O}(m)$ time.
	
	In the next section we will prove the following key lemma.
	
	\begin{lemma}\label{lem:key}
		$\mathcal{K}'$ has $O(n/\delta)$ vertices.
	\end{lemma}
	
	Theorem~\ref{thm:bb} follows directly from Lemma~\ref{lem:key}, as we have already observed that  $\mathcal{K}'$ can be computed in near-linear time.
	
	Assuming Lemma~\ref{lem:key}, we can deduce Theorem~\ref{thm:main} as follows. When we contract $\varphi'^{-1}(v)$ in $G$ for each cactus vertex $v$ of $\mathcal{K'}$, we obtain a graph with $|V(\mathcal{K}')| = O(n/\delta)$ vertices. The graph has at most $\lambda (|V(\mathcal{K}')| - 1) = O(n)$ edges by the definition of cactus representations, as the edges in $\mathcal{K'}$ can be covered by $|V(\mathcal{K}')| - 1$ uncrossing min-cuts (each of size $2$) and, consequently, the edges in $G$ that are not contracted can be covered by $|V(\mathcal{K}')| - 1$ min-cuts (each of size $\lambda$). Since $(\mathcal{K}', \varphi')$ is a cactus representation for $\mathcal{NC}(G)$, all non-trivial min-cuts are preserved after the contractions.
	
	This completes the proof of Theorem~\ref{thm:main}, and it remains to prove Lemma~\ref{lem:key}.

	\subsection{Proof of Lemma~\ref{lem:key}}
	We may assume that $\delta>6$, for otherwise, there is nothing to prove. In order to count the number of vertices of $\mathcal{K}'$, we consider a tree that reflects the cactus structure in a natural way.
	
	Given a cactus, a \emph{xylem} of the cactus is a tree that contains all cactus vertices plus one additional \emph{center} vertex for every cycle in the cactus, and has an edge between two vertices $v$ and $c$ if and only if $v$ is a cactus vertex and $c$ a center vertex such that $v$ is contained in the cycle represented by $c$ (see Figure~\ref{fig:example}). Note that the original edges of the cactus are not part of the xylem and that the xylem is indeed a tree, because the blocks of every cactus are its cycles and the blocks of every connected graph form a tree structure (the so-called \emph{block tree}). Since all cactus cycles have length at least two, the leaves of the xylem are exactly the 1-junction vertices of the cactus, i.e. the vertices of the cactus that appear in only one cycle. In order to prove the claim, we will prove the slightly stronger statement that the xylem $\mathcal{X}$ of $\mathcal{K}'$ has $O(n/\delta)$ vertices.
	
	Let $\mathcal{X'}$ be the tree obtained from $\mathcal{X}$ by deleting every vertex that is a 1-junction singleton in $\mathcal{K'}$. Note that the 1-junction non-singletons remain in $\mathcal{X'}$ and they are exactly the leaves of $\mathcal{X'}$. By Lemma~\ref{lem:no1junctionneighbors}, at most half of the vertices of every cycle in $\mathcal{K'}$ are 1-junction singletons. Hence, the leaf pruning reduces the number of vertices by at most factor two, and for the remainder it suffices to prove that $\mathcal{X}'$ has $O(n/\delta)$ vertices.
	
	We next argue that the leaf pruning in $\mathcal{X}$ does not create any new leaves.
	
	\begin{lemma}\label{lem:nocenterleaf}
		The set of leaves in $\mathcal{X'}$ is the set of 1-junction non-singletons in $\mathcal{K'}$.
	\end{lemma}
	\begin{proof}
		It suffices to show that no center vertex of $\mathcal{K'}$ is a leaf in $\mathcal{X'}$. Consider any $k$-cycle $A$ in $\mathcal{K'}$ and let $v$ be the center vertex that corresponds to that cycle. If $k = 2$, $A$ does not contain a 1-junction singleton by Modification~(i), so that $v$ is not a leaf. Otherwise, $k \geq 3$. By Lemma~\ref{lem:no1junctionneighbors}, $A$ contains at least two vertices that are not 1-junction singletons, so that $v$ is not a leaf.
	\end{proof}
	
	Consider any leaf $l$ in $\mathcal{X}'$. By Lemma~\ref{lem:nocenterleaf}, $l$ is a 1-junction vertex in $\mathcal{K}'$ and $|\varphi'^{-1}(l)| > 1$. By the definition of cactus representations, $\varphi'^{-1}(l)$ is a min-cut of $\mathcal{K}$ of size $\lambda \leq \delta$. Since $G$ is simple, this implies $|\varphi'^{-1}(l)| \geq \delta$ by a well-known lemma (see e.g.\ the proof of~\cite[Observation~1.5]{Kawarabayashi2018}).
	
	Thus, $\mathcal{X}'$ contains at most $O(n/\delta)$ leaves, and the fact that $\mathcal{X}'$ is a tree implies in turn that $\mathcal{X}'$ contains at most $O(n/\delta)$ vertices of degree at least three. It remains to show that $\mathcal{X}'$ contains $O(n/\delta)$ vertices of degree two.
	
	To see this, observe that $\mathcal{X'}$ is bipartite and every path is alternating between cactus vertices and center vertices. Consider any path $c_1$, $v_2$, $c_2$, $v_3$, $c_3$, $v_4$, $c_4$ of length six in $\mathcal{X}'$, where $c_i$ for every $i = 1,\dots,4$ is a center vertex and all internal vertices $v_2$, $c_2$, $v_3$, $c_3$, $v_4$ are of degree two in $\mathcal{X}'$. For $i = 1,\dots,4$, let $C_i$ be the cactus cycle in $\mathcal{K}'$ that corresponds to the center vertex $c_i$. For $i = 2,3,4$, $C_{i - 1}$ and $C_i$ are the two only cycles that contain $v_i$ in $\mathcal{K'}$, since $v_i$ has degree two in $\mathcal{X'}$ (and also in $\mathcal{X}$). Our assumption $\delta > 6$ together with Lemma~\ref{lem:no1junctionneighbors} imply that $C_i$ is a $k$-cycle with $2 \leq k \leq 4$ for $i = 2,3$, as $c_i$ is of degree two in $\mathcal{X'}$ and amongst its neighbors in $\mathcal{X}$ there are at most two 1-junction singletons to be deleted when constructing $\mathcal{X'}$ from $\mathcal{X}$. Furthermore, it is actually either a 2-cycle or a 4-cycle, since Modification~(ii) ensures that every 3-cycle in $\mathcal{K'}$ contains no 1-junction singleton.
	
	Let $W$ be the non-empty (by the existence of min-cuts) set of all vertices of $G$ that are mapped to a vertex in $V(C_2) \cup V(C_3)$ by $\varphi'$. Then $W$ is separated from $V-W \neq \emptyset$ by two min-cuts of $G$, one of which is represented by the two edges incident to $v_2$ in $C_1$ while the other is represented by the two edges incident to $v_4$ in $C_4$. Thus, at most $2\lambda$ edges leave $W$ in $G$. As $G$ is simple, we have $|W|(\delta - |W| + 1) \leq 2\lambda \leq 2\delta$. Together with the assumption $\delta > 6$, this implies that if $W$ has at least three vertices, $W$ has $\Omega(\delta)$ vertices.
	
	We next argue that $|W| \geq 3$, so that this $\Omega(\delta)$ lower bound is effective. Consider the cycle $C_i$ ($i = 2, 3$). If $C_i$ is a 2-cycle, then both $v_i$ and $v_{i + 1}$ are non-empty as ensured by Modification~(iii). If $C_i$ is a 4-cycle, then $C_i$ contains two non-empty 1-junction singletons that are different from $v_i$ and $v_{i+1}$. Thus, in every case, $V(C_2) \cup V(C_3)$ contains at least three non-empty cactus vertices, which implies $|W| \geq 3$. We conclude that $|W| \ge \delta - 1 = \Omega(\delta)$.
	
	A path in $\mathcal{X}'$ is called \emph{lean} if every vertex of it has degree two. By the last paragraph, $\mathcal{X}'$ contains $O(n/\delta)$ leaves, $O(n/\delta)$ vertices of degree at least three, and $O(n/\delta)$ maximal lean paths (they altogether partition the vertex set of $\mathcal{X'}$). We cut each maximal lean path into vertex-disjoint lean subpaths of length five, plus a potential remainder of a lean subpath of length less than five. Since every lean path of length five plus one of the two vertices adjacent to its ends form a path of length six as defined above, at least $\Omega(\delta)$ vertices of $G$ are mapped by $\varphi'$ to the vertices of every lean path of length five. Therefore, the number of these paths of length five is $O(n/\delta)$. Since the number of the remainder paths of length less than five is at most that of the maximal lean paths, i.e. $O(n/\delta)$, we conclude that $\mathcal{X}'$ and thus $\mathcal{K}'$ have $O(n/\delta)$ vertices. We remark that with a detailed calculation one can show that $\mathcal{K}'$ has less than $30n / \delta$ vertices.
	
	\subsection{Tightness}\label{sec:tightness}
	
	Let $n \geq 3(\delta+1)$, $\delta \geq 2$, $\lambda \leq \delta/2$ be an even positive integer, and assume that $n$ is a multiple of $\delta+1$ (the last assumption can be avoided by a simple modification of the construction). Set $r := n/(\delta+1)$. Let $G := (\{v_{i,j}: 1\leq i \leq r, 1\leq j \leq \delta+1\}, E_1 \cup E_2)$, where $E_1 := \bigcup_{i=1}^{r}\{v_{i,j} v_{i,k}: 1 \leq j < k \leq \delta + 1\}$ and $E_2 := \bigcup_{j=1}^{\lambda/2} \{v_{i, j} v_{i+1, j}: 1 \leq i \leq r\}$ (we set $v_{r+1, j} := v_{1, j}$ for all $1 \leq j \leq \delta + 1$). That is, $G$ is $r$ copies of cliques $K_{\delta+1}$ linked by $\lambda/2$ vertex-disjoint cycles of length $r$, with minimum degree $\delta$ and edge-connectivity $\lambda$. It is clear that $G$ shows tightness of Theorem~\ref{thm:main} and Corollary~\ref{cor:maincor}.
	 
	The assumption of connectedness in Corollary~\ref{cor:maincor} is (not only technically) necessary, as shown by the graph having $n/(\delta + 1)$ disjoint cliques $K_{\delta + 1}$, which has exponentially many non-trivial min-cuts if we fix $\delta$.
	
	\section{Enumerating all Min-Cuts}
	
	In this section we apply the compact cactus representation to a simple graph, to list all its min-cuts explicitly in $\tilde{O}(m) + O(n^2 / \delta)$ time, which proves Theorem~\ref{thm:enumeration}. This signicifantly improves the previous best time bound $O(nm)$ given by Gusfield and Naor~\cite{Gusfield1993}. Given a graph $G$, it is known that~\cite{Nagamochi1992} a subgraph $H \subseteq G$ with $O(n \delta)$ edges can be obtained in $O(m)$ time such that $H$ preserves all cuts of size not larger than $\delta$ and cuts of size larger than $\delta$ preserve at least $\delta$ of their edges; in particular, all min-cuts are preserved. Therefore we may assume that $m = \Theta(n\delta)$ when only the cuts of size not larger than $\delta$ concern us. From this point of view, we speeds up the approach of Gusfield and Naor by a factor of $\delta^2$.
	
	Intuitively, we may obtain a graph with $O(n / \delta)$ vertices and $O(n)$ edges preserving all non-trivial min-cuts by contracting the vertices of the cactus in our compact cactus representation, then apply the algorithm of Gusfield and Naor~\cite{Gusfield1993} to list all non-trivial min-cuts in $O(n^2 / \delta)$ time, and finally add the missing trivial min-cuts to the list if there are any. However, the graph obtained from contraction may have multiple edges. In this case we can not directly apply the result of Gusfield and Naor. To this end, we propose the following enumeration algorithm for a simple graph $G$.
	
	\begin{enumerate}
		\item Obtain a cactus representation $(\mathcal{K}_1, \varphi_1)$ for $\mathcal{NC}(G)$ in $\tilde{O}(m)$ time as described in Theorem~\ref{thm:bb}. 
	\end{enumerate}
	
	Let $H$ be the multigraph obtained from $G$ by contracting the vertices of $\mathcal{K}_1$. $H$ has $O(n /\delta)$ vertices and $O(n)$ edges, and preserves all non-trivial min-cuts of $G$.
	\begin{enumerate}[resume]
		\item Set $\mathcal{D}_1 := \emptyset$. For every 2-cycle $C$ in $\mathcal{K}_1$ with $V(C) = \{u, v\}$, we set $\mathcal{D}_1 := \mathcal{D}_1 \cup \{\mathcal{A}_C\}$, where $\mathcal{A}_C$ is the DAG with two vertices $\mathcal{K}_1[C, u], \mathcal{K}_1[C, v]$ and two edges both directed from $\mathcal{K}_1[C, u]$ to $\mathcal{K}_1[C, v]$. Contract all 2-cycles in $\mathcal{K}_1$ to obtain another cactus representation $(\mathcal{K}_2, \varphi_2)$. $\mathcal{D}_1$ and $(\mathcal{K}_2, \varphi_2)$ can be constructed in $O(n^2 / \delta)$ time as there are $O(n / \delta)$ 2-cycles in $\mathcal{K}_1$.
	\end{enumerate}
	
	It is clear that $\mathcal{A}_C$ has exactly one closed set which represents the min-cut $\mathcal{K}_1[C, v]$ in $\mathcal{K}_1$ and the min-cut $\varphi_1^{-1}(\mathcal{K}_1[C, v])$ in $G$. All other min-cuts represented by $(\mathcal{K}_1, \varphi_1)$ preserve if we contract $C$ in $\mathcal{K}_1$. Note that the cactus $\mathcal{K}_2$ is simple. 
	\begin{enumerate}[resume]
		\item Apply the method of Gusfield and Naor~\cite[Section 3.1]{Gusfield1993} to the cactus $\mathcal{K}_2$, to obtain a family $\mathcal{D}_2$ of $O(n / \delta)$ DAGs in $O(n^2 / \delta^2)$ time, such that every min-cut of $\mathcal{K}_2$ is exactly once represented by a closed set of some DAG from $\mathcal{D}_2$. Each of these DAGs $\mathcal{A}$ is associated with two distinct vertices $s_\mathcal{A}, t_\mathcal{A} \in V(\mathcal{K}_2)$ in the following way. Replace every edge $uv$ in $\mathcal{K}_2$ with directed edges $uv, vu$ orientated in opposite directions, find a maximum $s_\mathcal{A}$-$t_\mathcal{A}$-flow and construct the residual graph (which is called augmentation graph in~\cite{Gusfield1993}), finally contract all strongly connected components in the residual graph to obtain a DAG $\mathcal{A}$. 
	\end{enumerate}
	
	As $\mathcal{K}_2$ is a cactus graph, every $(\mathcal{A}, \rho) \in \mathcal{D}_2$ associated with $s_\mathcal{A}, t_\mathcal{A}$ can be derived in the following equivalent way. Let $\mathcal{K}_2^{(\mathcal{A})}$ be the minimal cactus subgraph of $\mathcal{K}_2$ containing $s_\mathcal{A}, t_\mathcal{A}$. It is not hard to see that $\mathcal{K}_2^{(\mathcal{A})}$ consists of cycles $C_1, C_2, \dots, C_k$ such that $s_\mathcal{A} \in V(C_1)$, $t_\mathcal{A} \in V(C_k)$, $s_\mathcal{A}$ and $t_\mathcal{A}$ are 1-junction vertices in $\mathcal{K}_2^{(\mathcal{A})}$, and $C_i$ intersects $C_{i + 1}$ with exactly one vertex for every $i = 1, \dots, k - 1$. The vertices of $\mathcal{K}_2^{(\mathcal{A})}$ have to be revised so that they form a partition of $V(\mathcal{K}_2)$. For every 1-junction vertex $v$ in $\mathcal{K}_2^{(\mathcal{A})}$, let $C$ be the cycle containing $v$, we set $v := \mathcal{K}_2[C, v]$ (i.e. $\rho$ maps every vertex in $\mathcal{K}_2[C, v]$ to $v$); for every 2-junction vertex $v$ in $\mathcal{K}_2^{(\mathcal{A})}$, let $C, C'$ be the cycles containing $v$, we set $v := \mathcal{K}_2[C, v] \cap \mathcal{K}_2[C', v]$. $\mathcal{A}$ is obtained by orientating the edges of $\mathcal{K}_2^{(\mathcal{A})}$ so that there are two edge-disjoint paths directed from $t_\mathcal{A}$ to $s_\mathcal{A}$, and $\mathcal{A}$ is exactly the union of these two paths. For ease of presentation, we see $\rho$ as a mapping from $V(\mathcal{K}_1)$ (instead of $V(\mathcal{K}_2)$) to $V(\mathcal{A})$ for every $(\mathcal{A}, \rho) \in \mathcal{D}_2$.
	
	We now have a collection $\mathcal{D} := \mathcal{D}_1 \cup \mathcal{D}_2$ of $O(n / \delta)$ DAGs such that their closed sets represent the min-cuts of $\mathcal{K}_1$ and this representation is a one-to-one correspondence. In order to let them represent the min-cuts of $G$ represented by $(\mathcal{K}_1, \varphi_1)$, we have to put back the edges of $H$ with orientation into the DAGs. 
	
	We claim that for every $(\mathcal{A}, \rho) \in \mathcal{D}$ and $uv \in E(H)$, if $\rho(\varphi_1(u)) \neq \rho(\varphi_1(v))$, then $\rho(\varphi_1(u))$ is either a predecessor or a successor of $\rho(\varphi_1(v))$ in $\mathcal{A}$. Suppose not, then there is a cycle $C$ in the underlying undirected graph of $\mathcal{A}$ such that $\rho(\varphi_1(u)), \rho(\varphi_1(v)) \in V(C)$ but they are not adjacent to each other. It implies that there is a cycle $C'$ in $\mathcal{K}_1$ such that $C'$ contains two non-adjacent vertices $u', v'$ satisfying $\varphi_1(u) \in \mathcal{K}_1[C', u']$, $\varphi_1(v) \in \mathcal{K}_1[C', v']$ and $d_G (\varphi_1^{-1}(\mathcal{K}_1[C', u']), \varphi_1^{-1}(\mathcal{K}_1[C', v'])) = d_H (\varphi_1^{-1}(\mathcal{K}_1[C', u']), \varphi_1^{-1}(\mathcal{K}_1[C', v'])) > 0$, which contradicts Lemma~\ref{lem:edgesbetweenneighbors}.
	
	\begin{enumerate}[resume]
		\item We aim for a collection $\mathcal{D}'$ of DAGs such that the closed sets of the DAGs from this collection represent the min-cuts of $G$. For every $(\mathcal{A}, \rho) \in \mathcal{D}$, we construct a DAG $\mathcal{A}'$ on the same vertex set $V(\mathcal{A})$ as follows. We first order the vertices of $\mathcal{A}$ linearly as $v_1 = t_\mathcal{A}, v_2, \dots, v_a = s_\mathcal{A}$ such that for any $1 \le i < j \le a$ we have either that $v_i$ is neither a predecessor nor a successor of $v_j$, or that $v_i$ is a predecessor of $v_j$ in $\mathcal{A}$. This can be readily done in $O(n / \delta)$ time by a topological sort. Set $V(\mathcal{A}') := V(\mathcal{A})$ and $E(\mathcal{A}') := \emptyset$. For every $uv \in E(H)$ with $\rho(\varphi_1(u)) \neq \rho(\varphi_1(v))$, say $\rho(\varphi_1(u)) = u_i$ and $\rho(\varphi_1(v)) = u_j$ for some $1 \le i < j \le a$, set $E(\mathcal{A}') := E(\mathcal{A}')  \cup \{u_i u_j\}$. Here we keep the multiple edges; more precisely, we identify the directed edge $u_i u_j \in E(\mathcal{A}')$ with the edge $uv \in E(H)$. It takes $O(n)$ time per DAG, and hence $\mathcal{D}' := \{\mathcal{A}' : \mathcal{A} \in \mathcal{D}\}$ can be constructed in $O(n^2 / \delta)$ time.
	\end{enumerate}
	
	We claim that for every $\emptyset \neq X \subset V(\mathcal{A}) = V(\mathcal{A}')$, $X$ is a closed set of $\mathcal{A}$ if and only if it is a closed set of $\mathcal{A}'$. It suffices to show that $u$ is a predecessor of $v$ in $\mathcal{A}$ if and only if $u$ is a predecessor of $v$ in $\mathcal{A}'$, equivalently, $u$ is a predecessor of $v$ in $\mathcal{A}$ for every $uv \in E(\mathcal{A}')$, and so in $\mathcal{A}'$ for every $uv \in E(\mathcal{A})$. That $u$ is a predecessor of $v$ in $\mathcal{A}$ for every $uv \in E(\mathcal{A}')$ follows immediately from the previous claim and the construction in Step~4. It is left to show that for every $uv \in E(\mathcal{A})$, $u$ is a predecessor of $v$ in $\mathcal{A}'$. Let $uv \in E(\mathcal{A})$, and $\mathcal{U}$ be the underlying undirected graph of $\mathcal{A}$. If both $u, v$ are 1-junction vertices in $\mathcal{U}$, then, by Lemma~\ref{lem:edgesbetweenneighbors}, exists a undirected edge $u' v' \in E(H)$ such that $\rho(\varphi_1(u')) = u$ and $\rho(\varphi_1(v')) = v$, and hence there exists an edge directed from $u$ to $v$ in $\mathcal{A}'$. If $u$ is 1-junction and $v$ is 2-junction in $\mathcal{U}$, let $C, C'$ be the cycles of $\mathcal{U}$ containing $v$, say $u$ is in $C$. Denote $X := \varphi_1^{-1}(\rho^{-1}(\{u\}))$, $Y := \varphi_1^{-1}(\rho^{-1}(\{v\}))$, $Z := \varphi_1^{-1}(\rho^{-1}(\mathcal{U}[C', v])) - X - Y$ and $W := \varphi_1^{-1}(\rho^{-1}(\mathcal{U}[C, v])) - Y$, i.e. $X, Y, Z$ and $W$ form a partition of $V(G)$. We recall that $\mathcal{K}_2 \supseteq \mathcal{U}$ is a simple graph, in particular, $|Z| \geq 1$ and $|W| > 1$. Indeed, the possibility that $|Z| = 1$ is ruled out by Modifications~(i) and~(ii), as $\mathcal{U}$ can be obtained from $\mathcal{K}_1$ by cycle contractions. Thus we have $|Z| > 1$. Suppose $d_G(X, Y) = 0$. By Lemma~\ref{lem:edgesbetweenneighbors}, we have $d_G(X, W) = d_G(X, Y \cup W) = \lambda / 2$, and hence $d_G(X, Y \cup Z) = d_G(X, Z) = \lambda / 2$. As $d_G(X \cup Y \cup Z) = \lambda$, we have $d_G(Y \cup Z) = d_G(X \cup Y \cup Z) - d_G(X, W) + d_G(X, Y \cup Z) = \lambda$. It implies that $\mathcal{U}$ has a cut of size 4 (the two edges incident to $u$ in $C$ and the two edges incident to $v$ in $C'$) which represents a non-trivial min-cut of $G$. But this non-trivial min-cut is not represented by any min-cut of $\mathcal{K}_1$, which contradicts that $(\mathcal{K}_1, \varphi_1)$ is a cactus representation for $\mathcal{NC}(G)$. Therefore $d_H(X, Y) = d_G(X, Y) > 0$, and there is an edge directed from $u$ to $v$ in $\mathcal{A}'$. Finally, if both $u, v$ are 2-junction vertices in $\mathcal{U}$, let $C$ be the cycle in $\mathcal{U}$ containing $u, v$. It is clear that there is a path of length at least 2 directed from $u$ to $v$ in $\mathcal{A}$ such that all internal vertices of this path are 1-junction in $\mathcal{U}$. By the previous two cases, we conclude that this directed path is also in $\mathcal{A}'$. In any case, we have $u$ is a predecessor of $v$ in $\mathcal{A}'$. Thus the claim is justified, and hence we can list all non-trivial min-cuts as follows.
	
	\begin{enumerate}[resume]
		\item Apply the method of Gusfield and Naor~\cite[Section 3.2]{Gusfield1993} to list, for every closed set of the DAGs of $\mathcal{D}'$, the $\lambda$ edges that go across it. We interpret them as the min-cuts of $H$ represented in $\mathcal{K}_1$. As it can be done in $O(\lambda)$ time per min-cut (closed set)~\cite{Gusfield1993}, and there are $O(n^2 / \delta^2)$ min-cuts in $H$, we list all min-cuts of $H$, including all non-trivial min-cuts of $G$, in $O(n^2 /\delta)$ time. At last we check if there are any missing trivial min-cuts of $G$ in $O(n + m)$ time.
	\end{enumerate}
	
	The correctness of the algorithm is clear. As the running time for Steps 1 to 5 is $\tilde{O}(m) + O(n^2 /\delta)$ in total, the proof of Theorem~\ref{thm:enumeration} is completed.
	 
	\bibliographystyle{abbrv}
	\bibliography{paper}

\end{document}